\documentclass{amsart}
\usepackage{tikz}
\usepackage{xcolor}
\usepackage{enumerate}
\date{\today}
\title[Favaron's Theorem and Tuza's Conjecture]{Favaron's Theorem, $k$-dependence, and Tuza's Conjecture}
\author{Gregory J.~Puleo}
\newcommand{\lz}{\ell}

\newcommand{\ints}{\mathbb{Z}}
\newcommand{\sey}{\mathcal{S}}
\newcommand{\lob}{\mathord{<}}
\newcommand{\dl}{d^{+}_J}
\newcommand{\dout}{d^+}

\newcommand{\dlp}{d^{+}_{J'}}

\newcommand{\join}{\vee}

\newcommand{\Xo}{S}

\newcommand{\sizeof}[1]{\left\lvert{#1}\right\rvert}

\newcommand{\st}{\colon\,}
\newcommand{\apk}{\alpha'_k}
\newtheorem{proposition}{Proposition}[section]
\newtheorem{conjecture}[proposition]{Conjecture}
\newtheorem{theorem}[proposition]{Theorem}
\newtheorem{lemma}[proposition]{Lemma}
\newtheorem{observation}[proposition]{Observation}
\newtheorem{corollary}[proposition]{Corollary}
\theoremstyle{definition}
\newtheorem{definition}[proposition]{Definition}
\theoremstyle{remark}
\newtheorem{remark}[proposition]{Remark}
\newcommand{\ind}{\alpha}
\begin{document}
\maketitle
\begin{abstract}
  A vertex set $D$ in a graph $G$ is \emph{$k$-dependent} if $G[D]$
  has maximum degree at most $k-1$, and \emph{$k$-dominating} if every
  vertex outside $D$ has at least $k$ neighbors in $D$. Favaron proved
  that if $D$ is a $k$-dependent set maximizing the quantity
  $k\sizeof{D} - \sizeof{E(G[D])}$, then $D$ is $k$-dominating.  We
  extend this result, showing that such sets satisfy a stronger
  structural property, and we find a surprising connection between
  Favaron's theorem and a conjecture of Tuza regarding packing
  and covering of triangles.
\end{abstract}
\section{Introduction}
A vertex set $D$ in a graph $G$ is \emph{independent} if the induced
subgraph $G[D]$ has no edges. A vertex set is \emph{dominating} if
every vertex of $G$ either lies in the set, or has a neighbor in the
set. Ore~\cite{Ore} observed that any maximal independent set is also
a dominating set: by the maximality of the independent set, every
vertex outside the set must have a neighbor in the set. Thus,
$\gamma(G) \leq \ind(G)$ for any graph $G$, where $\gamma(G)$ is the
size of a smallest dominating set and $\ind(G)$ is the size of a
largest independent set. Fink and Jacobson~\cite{FinkJacobson1,
  FinkJacobson2} generalized the notions of independence and
domination as follows.
\begin{definition}
  For positive integers $k$, a vertex set $D \subset V(G)$ is
  \emph{$k$-dependent} if the induced subgraph $G[D]$ has maximum
  degree at most $k-1$. A vertex set $D$ is \emph{$k$-dominating} if
  $\sizeof{N(v) \cap D} \geq k$ for all $v \in V(G) - D$.
\end{definition}
Fink and Jacobson posed the following question: letting $\gamma_k(G)$
denote the size of a smallest $k$-dominating set in $G$ and letting
$\ind_k(G)$ denote the size of a largest $k$-dependent set in $G$,
is it true that $\gamma_k(G) \leq \ind_k(G)$ for all $k$? Setting
$k=1$ yields the original inequality $\gamma(G) \leq \ind(G)$. However,
for $k>1$ it is no longer true that every maximal $k$-dependent set
is $k$-dominating. Favaron~\cite{Favaron} answered the question of
Fink and Jacobson, using a different notion of ``optimality''
for $k$-dependent sets.
\begin{theorem}[Favaron~\cite{Favaron}]\label{thm:favaron}
  If $D$ is a $k$-dependent set maximizing the quantity $k\sizeof{D} -
  \sizeof{E(G[D])}$ (over all $k$-dependent sets), then $D$ is a
  $k$-dominating set.
\end{theorem}
Since any set of at most $k$ vertices is a $k$-dependent set, it follows
that every graph has a set of vertices which is both $k$-dependent and
$k$-dominating, which yields $\gamma_k(G) \leq \ind_k(G)$.
Theorem~\ref{thm:favaron} motivates the following definition.
\begin{definition}
  For any set $D$, define $\phi_k(D) = k\sizeof{D} -
  \sizeof{E(G[D])}$. A \emph{$k$-optimal set} is a $k$-dependent set
  maximizing $\phi_k$ over all $k$-dependent sets.
\end{definition}
The notation $\phi_k(D)$ is borrowed from the survey paper
\cite{DomSurvey}. Notice that the maximum value of $\phi_k$ over all
$k$-dependent sets is equal to the maximum value of $\phi_k$ over all
vertex sets, since if $D$ is an arbitrary vertex set and $v \in D$ has
degree exceeding $k-1$ in $G[D]$, then $\phi_k(D-v) \geq
\phi_k(D)$. We extend Theorem~\ref{thm:favaron} by proving that
$k$-optimal sets satisfy a property stronger than $k$-domination.
\begin{theorem}\label{thm:main}
  Let $D$ be a $k$-optimal set in a graph $G$, let $X = V(G)-D$, and
  let $H$ be the maximal bipartite subgraph of $G$ with partite sets
  $D$ and $X$. If $J$ is any orientation of $G[X]$, then $H$ has a
  $k$-edge-chromatic subgraph $M$ such that $d_M(v) + \dl(v) \geq k$
  for all $v \in X$, where $\dl(v)$ is the outdegree of $v$ in the
  orientation $J$.
\end{theorem}
In particular, since we can take any vertex $v \in V(G)-D$ to have
outdegree $0$ in $J$, Theorem~\ref{thm:main} implies that any $k$-optimal
set is $k$-dominating. In fact, we have the following stronger
corollary, obtained by taking the vertices of an independent set to have
outdegree $0$ in $J$.
\begin{corollary}
  Let $D$ be a $k$-optimal set in a graph $G$. For any independent set $S$
  disjoint from $D$, there are $k$ disjoint matchings of $S$ into $D$, each
  saturating $S$.
\end{corollary}
We now turn to a conjecture of Tuza concerning packing and covering of triangles.
Given a graph $G$, let $\tau(G)$ denote the minimum size of an edge set $Y$
such that $G-Y$ is triangle-free, and let $\nu(G)$ denote the maximum size of
a set of pairwise edge-disjoint triangles in $G$. It is easy to show that
$\nu(G) \leq \tau(G) \leq 3\nu(G)$: if $\sey$ is a largest set of pairwise
edge-disjoint triangles, then to make $G$ triangle-free we must delete at least
one edge from each triangle of $\sey$, and on the other hand deleting all edges
contained in triangles of $\sey$ will always make $G$ triangle-free. Tuza
conjectured a stronger upper bound.
\begin{conjecture}[Tuza's Conjecture~\cite{TuzaProc,Tuza}]
  $\tau(G) \leq 2\nu(G)$ for all graphs $G$.
\end{conjecture}
Tuza's Conjecture is sharp, if true; as observed by Tuza~\cite{Tuza},
equality in the upper bound is acheieved by any graph whose blocks are
all isomorphic to $K_4$, among other examples. The best general upper
bound on $\tau(G)$ in terms of $\nu(G)$ is due to
Haxell~\cite{Haxell}, who showed that $\tau(G) \leq 2.87\nu(G)$ for
all graphs $G$. Tuza's Conjecture has been studied by many 
authors, who proved the conjecture for special classes of graph
\cite{SashaK4,Krivelevich,Puleo,LBT,LBT-perfect} or studied various
fractional relaxations of the conjecture
\cite{CDMMS,HaxellRodl,SashaStability,Krivelevich}.

A major theme of the author's previous work on Tuza's
Conjecture~\cite{Puleo} is to reduce questions about triangle packings to questions about
matchings, since matchings are very well understood. In this paper,
we therefore study the conjecture on graphs of the form $I_k \join H$,
where $I_k$ is an independent set of size $k$, $H$ is a triangle-free graph, and the
\emph{join} $G_1 \join G_2$ of two graphs $G_1$ and $G_2$ is obtained
from the disjoint union of $G_1$ and $G_2$ by adding all possible
edges between $V(G_1)$ and $V(G_2)$. Each triangle of $I_k \join H$
consists of an edge in $H$ together with a vertex of $I$; thus,
triangle packings in $I_k \join H$ correspond to partial
$k$-edge-colorings of $H$. The connection between Tuza's~Conjecture
and Favaron's~Theorem is given by the following result.
\begin{theorem}\label{thm:tuzaconnection}
  For a graph $G$ and $k \in \ints^+$, let $\phi_k(G)$ denote the
  maximum value of $\phi_k(D)$ over all $k$-dependent sets $D \subset
  V(G)$, and let $\alpha'_k(G)$ denote the largest number of edges in
  a $k$-edge-colorable subgraph of $G$. If $H$ is triangle-free, then
  \begin{align*}
    \nu(I_k \join H) &= \alpha'_k(H),\text{ and} \\
    \tau(I_k \join H) &= k\sizeof{V(H)} - \phi_k(H).
  \end{align*}
\end{theorem}
Thus, this special case of Tuza's Conjecture can be interpreted as requesting a
relationship between the $\phi_k$-value of a $k$-optimal set and the size of a largest
$k$-edge-colorable subgraph of a graph. We formalize this special case below.
\begin{conjecture}\label{coj:specialtuza}
  If $k \in \ints^+$ and $H$ is a triangle-free graph, then $\tau(I_k \join H) \leq 2\nu(I_k \join H)$.
  Equivalently, $\alpha'_k(H) \geq k\sizeof{V(H)} - \phi_k(H)$.
\end{conjecture}
A similar idea appears in a paper of Chapuy, DeVos, McDonald, Mohar,
and Schiede \cite{CDMMS}.  Studying a fractional version of Tuza's
conjecture, they consider a triangle-free Ramsey graph $H$ with low
independence number (hence matching number close to $\sizeof{V(H)}/2$) and use
$I_1 \join H$ as a sharpness example for an upper bound on $\tau(G)$.

The rest of the paper is organized as follows. In Section~\ref{sec:k1}
we build some intuition by exploring the $k=1$ case of
Theorem~\ref{thm:main}.  Based on a result that holds in the $k=1$
case, we pose the following conjecture, which is natural in its own right
and which implies Conjecture~\ref{coj:specialtuza}:
\begin{conjecture}\label{coj:sec1deg}
  If $D$ is a $k$-optimal set in a graph $G$, then $G$ has a
  $k$-edge-chromatic subgraph in which every vertex of $V(G)-D$
  has degree $k$.
\end{conjecture}
In Section~\ref{sec:leben} we prove a generalization of a theorem of
Lebensold concerning disjoint matchings in bipartite graphs. In
Section~\ref{sec:proof} we apply the results of
Section~\ref{sec:leben} to prove Theorem~\ref{thm:main}. In
Section~\ref{sec:tuza} we prove Theorem~\ref{thm:tuzaconnection} and
and show that Conjecture~\ref{coj:sec1deg} implies
Conjecture~\ref{coj:specialtuza}. In Section~\ref{sec:chordal} we
introduce the notion of \emph{saturability}, which is a variant of
list edge coloring. We formulate a conjecture about graph
decompositions and show that this conjecture implies
Conjecture~\ref{coj:sec1deg}.  The proof uses Galvin's kernel
method~\cite{Galvin}. Finally, we apply Theorem~\ref{thm:main} to
prove that Conjecture~\ref{coj:sec1deg} holds for chordal graphs.

\section{The $k=1$ Case of Theorem~\ref{thm:main}}\label{sec:k1}
When $k=1$, things are simpler: a $1$-dependent set is just an
independent set, so a $1$-optimal set is just a maximum-size
independent set. The statement of Theorem~\ref{thm:main} can also be
simplified: when $k=1$, the only vertices of $V(G)-D$ for which the
theorem says anything are those of outdegree $0$, which form an
independent set. Thus, the $k=1$ case of Theorem~\ref{thm:main} is
equivalent to the following proposition, which we prove using
Hall's~Theorem~\cite{HallsTheorem}:
\begin{proposition}\label{prop:k1}
  Let $D$ be a maximum-size independent set in a graph $G$. If $T$ is
  an independent set disjoint from $D$, then $G$ has a matching of $T$
  into $D$ that saturates $T$.
\end{proposition}
\begin{proof} 
  Assume that no such matching exists; we will obtain an independent
  set $D'$ with $\sizeof{D'} > \sizeof{D}$. If no such matching
  exists, then by Hall's Theorem, there is a set $S \subset T$ such
  that $\sizeof{N(S) \cap D} < \sizeof{S}$. Let $D' = (D - N(S)) \cup S$.
  Since $S$ and $D$ are independent and since we have deleted $N(S)$,
  the set $D'$ is also independent. Since $\sizeof{N(S) \cap D} < \sizeof{S}$,
  we have $\sizeof{D'} > \sizeof{D}$, as desired.
\end{proof}
In fact, Proposition~\ref{prop:k1} is one half of Berge's
``Maximum Stable Set Lemma''~\cite{Berge}, which states that an
independent set $D$ is maximum if and only if every independent set
$T$ disjoint from $D$ matches into $D$.

The proof of Theorem~\ref{thm:main} uses a similar strategy to the
proof of Proposition~\ref{prop:k1}: assuming that the desired
$k$-edge-chromatic subgraph does not exist, we obtain a set of
vertices for which a similar Hall-type condition fails, and we use
this set of vertices to construct a ``better'' $k$-dependent set.

Proposition~\ref{prop:k1} implies the $k=1$ case of
Conjecture~\ref{coj:sec1deg}, but we have not been able to generalize
this proof to higher $k$.
\begin{corollary}\label{cor:match}
  If $D$ is a maximum-size independent set in a graph $G$, then
  $G$ has a matching that saturates every vertex of $V(G)-D$.
\end{corollary}
\begin{proof}
  Let $M_1$ be a maximal matching in $V(G)-D$, and let $S$ be
  the set of vertices in $V(G)-D$ not saturated by $M_1$. Since
  $M_1$ is a maximal matching, $S$ is an independent set. By
  Proposition~\ref{prop:k1}, there is a matching $M_2$ of $S$
  into $D$ that saturates $S$. Thus, $M_1 \cup M_2$ is a matching
  that saturates $V(G)-D$.
\end{proof}
 Since Theorem~\ref{thm:main} generalizes
Proposition~\ref{prop:k1}, one would hope that it could play an
analogous role in a proof of Conjecture~\ref{coj:sec1deg}, but it
appears difficult to find a suitable orientation to use. This difficulty
is explored further in Section~\ref{sec:tuza}.
\section{Extending Lebensold's Theorem}\label{sec:leben}
Lebensold~\cite{Lebensold} proved the following generalization of
Hall's~Theorem. As Brualdi observed in his review of \cite{Lebensold},
the theorem is equivalent to a theorem of Fulkerson~\cite{Fulkerson}
concerning disjoint permutations in $0,1$-matrices.  An alternative
proof of the theorem, using matroid theory, is due to
Murty~\cite{Murty}.
\begin{theorem}[Lebensold~\cite{Lebensold}]\label{thm:leb}
  An $X,D$-bigraph has $k$ disjoint matchings from $X$ into $D$, each saturating $X$, if and only if
  \[ \sum_{v \in D}\min\{ k, \sizeof{N(v) \cap \Xo} \} \geq k\sizeof{\Xo} \]
  for every subset $\Xo \subset X$.
\end{theorem}
We extend the theorem to find necessary and sufficient conditions for the existence
of a $k$-edge-chromatic subgraph in which the vertices of $X$ are allowed to have
different degrees.
\begin{lemma}\label{lem:genleb}
  Let $H$ be an $X,D$-bigraph, and write $X = \{v_1, \ldots,
  v_t\}$. Let $k$ be a positive integer and let $d_1, \ldots, d_t$ be
  nonnegative integers with all $d_i \leq k$. The following are
  equivalent:
  \begin{enumerate}[(1)]
  \item $H$ has a $k$-edge-chromatic subgraph $M$ such that $d_M(v_i) \geq d_i$ for all $i$;
  \item For every subset $\Xo \subset X$,
    \[ \sum_{v \in D}\min\{ k, \sizeof{N(v) \cap \Xo}\} \geq \sum_{v_i \in \Xo}d_i. \]
  \end{enumerate}
\end{lemma}
Theorem~\ref{thm:leb} is the special case of Lemma~\ref{lem:genleb}
obtained when all $d_i = k$.  We prove Lemma~\ref{lem:genleb} using
Theorem~\ref{thm:leb}, so Theorem~\ref{thm:leb} is
self-strengthening in this sense.
\begin{proof}
  For each $i$, let $D_i$ be a set of size $k-d_i$, with all sets
  $D_i$ disjoint from each other and disjoint from $V(H)$, and let $D'
  = D \cup D_1 \cup \cdots \cup D_t$. Let $H'$ be the $X,D'$-bigraph
  obtained from $H$ by making the vertices in $D_i$ adjacent only to
  $v_i$. Consider the following two statements:
  \begin{enumerate}[(1$'$)]
  \item $H'$ has $k$ edge-disjoint matchings, each saturating $X$;
  \item For every subset $\Xo \subset X$,
    \[ \sum_{v \in D'}\min\{ k, \sizeof{N(v) \cap \Xo}\} \geq k\sizeof{\Xo}. \]
  \end{enumerate}
  By Theorem~\ref{thm:leb}, (1$'$) is equivalent to (2$'$). We prove that
  (1) is equivalent to (1$'$) and (2) is equivalent to (2$'$).

  If $M_1, \ldots, M_k$ are edge-disjoint matchings in $H'$ each
  saturating $X$, then their restriction to $H$ yields a
  $k$-edge-chromatic subgraph $M$ of $H$ with each $d_M(v_i) \geq
  d_i$.  Conversely, any such subgraph of $H$ can be extended to $k$
  edge-disjoint matchings in $H'$. Thus, (1) is equivalent to (1$'$).

  Elements of $D_i$ each contribute $1$ to the sum in (2$'$) when $v_i \in \Xo$, and contribute
  $0$ otherwise. This yields
  \[ \sum_{v \in D'}\min\{k, \sizeof{N(v) \cap \Xo}\} = \sum_{v_i \in \Xo}(k - d_i) + \sum_{v \in D}\min\{k, \sizeof{N(v) \cap \Xo}\}, \]
  so (2) is equivalent to (2$'$).
\end{proof}
\section{Proof of Theorem~\ref{thm:main}}\label{sec:proof}
The following observation was mentioned in the introduction; we repeat
it here because of its central role in this section.
\begin{observation}\label{obs:subset}
  If $G$ is a graph and $T \subset V(G)$, then for any $k$, there is a
  $k$-dependent subset $D \subset T$ such that $\phi_k(D) \geq \phi_k(T)$.
\end{observation}
As a consequence of Observation~\ref{obs:subset}, in order to prove that
some vertex set $D$ is not $k$-optimal, it suffices to find \emph{any}
vertex set $D'$ with $\phi_k(D') > \phi_k(D)$, without worrying about
whether $D'$ is $k$-dependent. When $k$ is a nonnegative
integer, we write $[k]$ for the set $\{1, \ldots, k\}$.

We are now ready to prove Theorem~\ref{thm:main}.
\begin{proof}[Proof of Theorem~\ref{thm:main}]
  Let $D$ be a $k$-dependent set, let $X = V(G) - D$, and let $J$ be
  an orientation of $G[X]$. Assuming that there is no
  $k$-edge-chromatic subgraph with the desired properties, we
  construct a set $D'$ with $\phi_k(D') > \phi_k(D)$.

  Since there is no $k$-edge-chromatic subgraph with the desired
  properties, applying Lemma~\ref{lem:genleb} with $d_i = \max\{0, k -
  \dl(v)\}$ shows that there is a set $\Xo \subset X$ such that
  \[ \sum_{v \in D}\min\{ k, \sizeof{N(v) \cap \Xo}\} < \sum_{v_i \in \Xo}\max\{0, k-\dl(v)\}. \]
  We may assume that $\dl(v) \leq k$ for all $v \in \Xo$, since vertices with $\dl(v) > k$
  may be removed from $\Xo$ without causing the above inequality to fail. This gives the
  simpler inequality
  \begin{equation}
    \label{eq:lebfail}
    \sum_{v \in D}\min\{ k, \sizeof{N(v) \cap \Xo}\} < k\sizeof{\Xo} - \sum_{v_i \in \Xo}\dl(v).
  \end{equation}
Define sets $B$ and $A$ by
\begin{align*}
 B &= \{v \in D \st \sizeof{N(v) \cap \Xo} \leq k-1\}, \\
 A &= D-B, \\
\end{align*}
Let $D' = B \cup \Xo$. We claim that
$\phi_k(D') > \phi_k(D)$. 
First observe that
\[ \sum_{v \in D}\min\{ k, \sizeof{N(v) \cap \Xo}\} = k\sizeof{A} + \sum_{v \in B}\sizeof{N(v) \cap \Xo} = k\sizeof{A} + \sizeof{[\Xo, B]}, \]
where $[\Xo, B]$ is the set of all edges with one endpoint in $\Xo$ and the other endpoint in $B$.
Thus, from \eqref{eq:lebfail},
\begin{align}
  k\sizeof{A} + \sizeof{[\Xo, B]} &< k\sizeof{\Xo} - \sum_{v_i \in \Xo}\dl(v)\nonumber\\
  &\leq k\sizeof{\Xo} - \sizeof{E(G[\Xo])}.\label{eq:cat}
\end{align}
On the other hand, we have
\[
  \phi_k(D') \geq \phi_k(D) - k\sizeof{A} + k\sizeof{\Xo} - \sizeof{[\Xo, B]} - \sizeof{E(G[\Xo])}.
\]
Combining this with Inequality~\ref{eq:cat} yields $\phi_k(D') > \phi_k(D)$, as desired.
Thus, when $D$ is $k$-optimal, a $k$-edge-chromatic subgraph with the desired properties exists.
\end{proof}
\section{Tuza's Conjecture and Conjecture~\ref{coj:sec1deg}}\label{sec:tuza}
We first prove Theorem~\ref{thm:tuzaconnection}, which furnishes a connection
between Tuza's~Conjecture and $k$-optimal sets.
\begin{proof}[Proof of Theorem~\ref{thm:tuzaconnection}]
  Let $G = I_k \join H$.  We first show that $\nu(G) = \apk(H)$. Let
  $\sey$ be a maximum set of edge disjoint triangles in $G$.  For each
  $v \in I_k$, let $\sey_v = \{T \in \sey \st v \in T\}$. Since each
  triangle in $G$ consists of exactly one vertex of $I_k$ together
  with an edge in $H$, we can write $\sey$ as the disjoint union $\sey
  = \bigcup_{v \in I_k}\sey_v$. Since the triangles in $\sey_v$ are
  edge-disjoint, no two triangles in $\sey_v$ can share a common
  vertex $w \in V(H)$: if this were the case, they would intersect in
  the edge $vw$. Hence the edges of $\sey_v$ that lie in $H$ form a
  matching $M_v$ in $H$. Since the triangles in $\sey$ are
  edge-disjoint, it follows that the matchings $M_v$ are pairwise
  disjoint, so $\bigcup_{v \in I_k}M_v$ is a $k$-edge-colorable
  subgraph of $H$ having size $\nu(G)$. Therefore, $\nu(G) \leq
  \apk(H)$.

  On the other hand, if $H_0$ is a maximum $k$-edge-colorable subgraph of $H$,
  then we can write $E(H_0) = M_1 \cup \cdots \cup M_k$, where each $M_i$ is
  a matching. Let $v_1, \ldots, v_k$ be the vertices of $I_k$, and for $i \in [k]$,
  let $\sey_i = \{v_iwz \st wz \in M_i\}$. Now $\bigcup_{i \in [k]}\sey_i$ is a family
  of $\apk(H)$ pairwise edge-disjoint triangles in $G$, so $\nu(G) \geq \apk(H)$.

  Next we show that $\tau(G) = k\sizeof{V(H)} - \phi_k(H)$. Let $D$ be a $k$-optimal
  subset of $V(H)$, and define an edge set $X$ by
  \[ X = E(D) \cup \{vw \st \text{$v \in I_k$ and $w \in
    V(H)-D$}\}. \] Clearly, $\sizeof{X} = \sizeof{E(G[D])} +
  k(\sizeof{V(H)}-\sizeof{D})$, which rearranges to $\sizeof{X} =
  k\sizeof{V(H)} - \phi_k(H)$, since $D$ is $k$-optimal.  We claim that
  $G-X$ is triangle-free. Let $T$ be any triangle in $G$; we may write
  $T=uvw$, where $uw \in E(H)$ and $v \in I_k$. If $u \notin D$, then
  $vu \in X$, and likewise for $w$. On the other hand, if $u,w \in D$,
  then $uw \in X$. Hence $G-X$ is triangle-free, so $\tau(G) \leq k\sizeof{V(H)} - \phi_k(H)$.

  Conversely, let $X$ be a minimum edge set such that $G-X$ is
  triangle-free. For each $v \in I_k$, let $C_v = \{w \in V(H) \st vw
  \notin X\}$.  We transform $X$ so that all the sets $C_v$ are equal:
  pick $v^* \in I_k$ to minimize $\sizeof{C_{v^*}}$, and define $X_1$
  by
  \[ X_1 = (X \cap E(H)) \cup \{vw \st w \in C_{v^*}\}. \]
  Now $G-X_1$ is triangle-free: if $vwz$ is a triangle in $G-X_1$, then $v^*wz$
  is a triangle in $G-X$, contradicting the assumption that $G-X$ is
  triangle-free. Furthermore, by the minimality of $\sizeof{C_{v^*}}$, we
  have $\sizeof{X_1} \leq \sizeof{X}$.

  Therefore, $\sizeof{X_1} = \sizeof{X} = \tau(G)$. Let $D = V(H) - C_{v^*}$.
  Since $G-X_1$ is triangle-free, we have $E(G[D]) \subset X_1$, and so
  \[ \sizeof{X_1} \geq \sizeof{E(G[D])} + k\sizeof{V(H) - D} = k\sizeof{V(H)} - \phi_k(D). \]
  While $D$ need not be $k$-dependent, the maximum value of $\phi_k$ over subsets of $H$
  is achieved at some $k$-dependent vertex set, so $\phi_k(D) \leq \phi_k(H)$. Hence,
  \[ \tau(G) = \sizeof{X_1} \geq k\sizeof{V(H)} - \phi_k(H).\qedhere\]
\end{proof}
The following lemma shows that Conjecture~\ref{coj:sec1deg} implies Conjecture~\ref{coj:specialtuza}.
\begin{lemma}\label{lem:edgetuza}
  Let $D$ be a $k$-optimal set in a graph $G$. If $G$ has
  a $k$-edge-chromatic subgraph in which every vertex of $V(G)-D$ has
  degree $k$, then
  \[ k\sizeof{V(G)} - \phi_k(G) \leq 2\apk(G). \]
  In particular, if $G$ is triangle-free, then $\tau(I_k \join G) \leq 2\nu(I_k \join G)$.
\end{lemma}
\begin{proof}
  Since $\Delta(G[D]) \leq k-1$, Vizing's~Theorem~\cite{Vizing} implies that $G[D]$ is $k$-edge-colorable.  Let
  $N_1, \ldots, N_k$ be the $k$ edge-disjoint matchings in a
  $k$-edge-coloring of $G[D]$, so that $N_1 \cup \cdots \cup N_k = E(G[D])$, and let $M'_1, \ldots, M'_k$ be $k$
  pairwise edge-disjoint matchings each saturating $V(G)-D$, as guaranteed by the hypothesis.
  For each $i$, let $M'_i$ be the matching defined by
  \[ M_i = M'_i \cup \{e \in N_i \st \text{$e \cap f = \emptyset$ for all $f \in M'_i$}\}. \]
  Let $T = M_1 \cup \cdots \cup M_k$. Clearly $T$ is $k$-edge-chromatic, so $\apk(G) \geq \sizeof{T}$.

  A \emph{cross-edge} is an edge of $T$ with exactly one endpoint in $D$.
  If there are $q$ cross-edges, then the degree-sum formula yields
  \[ \sum_{i=1}^k \sizeof{M'_i} = \frac{k}{2}\sizeof{V(G)-D} + \frac{q}{2}. \]
  Furthermore, if there are $r$ edges of $G[D]$ which fail to appear in $T$,
  then each of these edges is ``witnessed'' by a different cross-edge, so $r \leq q$.
  Clearly $r \leq \sizeof{E(G[D])}$, so
  \begin{align*}
    \sum_{i=1}^k \sizeof{M_i} &= \frac{k}{2}\sizeof{V(G) - D} + \frac{q}{2} + \sizeof{E(G[D])} - r \\
    &\geq \frac{k}{2}\sizeof{V(G) - D} + \sizeof{E(G[D])} - \frac{r}{2} \\
    &\geq \frac{k}{2}\sizeof{V(G) - D} + \frac{1}{2}\sizeof{E(G[D])}.
  \end{align*}  
  It follows that
  \[ 2\apk(G) \geq 2\sizeof{T} \geq k\sizeof{V(G)-D} + \sizeof{E(G[D])} = k\sizeof{V(G)} - \phi_k(G). \qedhere \]  
\end{proof}
\begin{corollary}
  Conjecture~\ref{coj:sec1deg} implies Conjecture~\ref{coj:specialtuza}.
\end{corollary}
\begin{remark}
  The proof technique of Lemma~\ref{lem:edgetuza} \emph{almost} yields
  an inductive proof of Conjecture~\ref{coj:sec1deg}, in the following
  sense.

  Let $D$ be a $k$-optimal set of a graph $G$, let $X = V(G)-D$, and
  let $G' = G[X]$. Suppose that Conjecture~\ref{coj:sec1deg} holds for
  $G'$, let $D'$ be a $k$-optimal set in $G[X]$, and let
  $M'_1, \ldots, M'_k$ be $k$ disjoint matchings in $G'$, each saturating
  $X-D'$.

  Let $N_1,\ldots,N_k$ be the matchings of a $k$-edge-coloring of
  $G[D]$, and use $N_1,\ldots,N_k$ and $M'_1,\ldots,M'_k$ to form a
  $k$-edge-chromatic subgraph $T'$ in $G'$ by the same method used in
  the proof of Lemma~\ref{lem:edgetuza}. Now for each edge $uv$ of $G[D']$
  that fails to lie in $T'$, we have $uv \in N_i$ for some $i$,
  and without loss of generality, the endpoint $u$ is covered by $M'_i$.
  Orient $uv$ from $u$ to $v$. (If both endpoints are covered by $M'_i$,
  we may orient $uv$ either way.) Let $J$ be the resulting orientation
  of $G[D']$.

  Observe that each vertex $v \in D'$ is incident to $\dl(v)$ edges
  of the $k$-edge-chromatic graph $T'$. Since $D$ is $k$-optimal in
  $G$, it is $k$-optimal in $G[D \cup D']$. Now by
  Theorem~\ref{thm:main}, there is another $k$-edge-chromatic graph
  $M$ with $E(M) \subset [D, D']$ such that $d_M(v) + \dl(v) \geq k$
  for each $v \in D'$. Thus $M \cup T'$ is a subgraph of $G$ where all
  vertices outside $D$ have degree at least $k$; indeed, $M \cup T'$
  can be taken so that the vertices outside $D$ have degree exactly
  $k$ and the vertices of $D'$ have degree at most $k$. The problem is
  that $M \cup T'$ need not be $k$-edge-chromatic, even though $M$ and
  $T'$ are $k$-edge-chromatic. In other words, we cannot guarantee
  that the ``received coloring'' of $M$ agrees with the received
  coloring of $T'$.
\end{remark}
\section{A Graph Decomposition Conjecture}\label{sec:chordal}
In this section, we introduce a conjecture concerning graph
decomposition which implies Conjecture~\ref{coj:sec1deg} (and
therefore Conjecture~\ref{coj:specialtuza}).
\begin{definition}
  A \emph{sequential decomposition} of a digraph $J$ is a sequence $H_1, H_2, \ldots$
  of spanning subdigraphs of $J$ such that each arc of $J$ belongs to exactly one
  $H_i$.
\end{definition}
The only difference between a sequential decomposition and a decomposition in the
usual sense is that a sequential decomposition has an explicit order on the subdigraphs,
whereas a classical decomposition is an unordered collection of subdigraphs whose edges
partition $E(J)$.
\begin{conjecture}\label{coj:decomp}
  For every graph $G$, there is an orientation $J$ having no directed
  odd cycle and admitting a sequential decomposition
  $H_1, H_2, \ldots$ such that:
  \begin{itemize}
  \item Each component of $H_i$ is a directed path or directed even cycle, and
  \item For each $v \in V(J)$ and each $i$, we have $\dout_{H_i}(v) \geq \dout_{H_{i+1}}(v)$.
  \end{itemize}
\end{conjecture}
We call a sequential decomposition of the form requested by
Conjecture~\ref{coj:decomp} a \emph{good decomposition} of $J$.  In
the rest of this section, we prove that Conjecture~\ref{coj:decomp}
implies Conjecture~\ref{coj:sec1deg}, and we prove
Conjecture~\ref{coj:decomp} for chordal graphs, which implies the
following special case of Conjecture~\ref{coj:sec1deg}.
\begin{theorem}\label{thm:chordal}
  If $D$ is a $k$-optimal set in a chordal graph $G$, then $G$ has a $k$-edge-chromatic
  subgraph in which every vertex of $V(G)-D$ has degree $k$.
\end{theorem}
The proof requires some new definitions. Our main definition is a
variant on list coloring. While the concept of a list assignment is
well-studied in other contexts, the concept of an $\lz$-saturating
partial edge coloring for a list assignment is, to our knowledge,
new. (Typically, list assignments are used in vertex coloring
problems, and represent the colors available at a vertex.)
\begin{definition}\label{def:saturate}
  A \emph{list assignment} on a graph $G$ is a function $\lz$ that
  assigns to each vertex $v$ a set of colors $\lz(v)$. A (proper)
  partial edge coloring is \emph{$\lz$-saturating} if for every vertex
  $v$ and every color $c \in \lz(v)$, some edge incident to $v$ receives
  the color $c$. A graph $G$ is \emph{$\lz$-saturable} if it has an
  $\lz$-saturating partial edge coloring.  When $f$ is a function from
  $V(G)$ into the nonnegative integers, we say that $G$ is
  \emph{$f$-saturable} if $G$ is $\lz$-saturable whenever $\sizeof{\lz(v)}
  \leq f(v)$ for all $v$.
\end{definition}
The connection between Conjecture~\ref{coj:sec1deg} and Definition~\ref{def:saturate}
is given by the following lemma.
\begin{lemma}\label{lem:satur}
  Let $D$ be a $k$-optimal set in a graph $G$, and let $X =
  V(G)-D$. If $G[X]$ has some orientation $J$ such that $G[X]$ is
  $\dl$-saturable, then $G$ has a $k$-edge-chromatic subgraph in which
  all vertices of $X$ have degree $k$.
\end{lemma}
\begin{proof}
  Let $J$ be any such orientation of $X$, and let $M'$ be the
  $k$-edge-chromatic subgraph given by Theorem~\ref{thm:main}. Write
  $M' = M'_1 \cup \cdots \cup M'_k$, where each $M_i$ is a matching,
  and for each $v \in X$, let $\lz(v)$ be the set of all indices
  $i \in [k]$ such that $v$ is not saturated by $M'_i$. The degree
  condition on $M'$ yields $\sizeof{\lz(v)} = \min\{\dl(v), k\}$.

  Since $G[X]$ is $\dl$-saturable, there is a partial edge coloring
  $\psi$ of $G[X]$ that is $\lz$-saturating. For $i \in [k]$, let $M^*_i$
  be the set of edges that receive color $i$ in $\psi$. Now we combine the
  matchings $M_1, \ldots, M_k$ with the matchings $M^*_1, \ldots, M^*_k$.
  For each $i \in [k]$, let $M_i$ be the set defined by
  \[ M_i = M^*_i \cup \{e \in M'_i \st \text{$e \cap e^* = \emptyset$
    for all $e^* \in M^*_i$}\}. \] By construction, each edge set $M_i$
  is a matching. Furthermore, every vertex of $v$ is incident to an
  edge in each $M_i$, since if $v$ is not incident to any edge of
  $M'_i$, then $i \in L(v)$, so that $v$ is incident to an edge in
  $M^*_i$. Thus, the $k$-edge-chromatic subgraph $M_1 \cup \cdots \cup
  M_k$ has the desired properties.
\end{proof}
\begin{lemma}\label{lem:decomp-sat}
  Let $J$ be an orientation of a graph $G$. If $J$ admits a good decomposition,
  then $G$ is $\dl$-saturable.
\end{lemma}
The proof of Lemma~\ref{lem:decomp-sat} uses Galvin's kernel method; we reproduce
the relevant definitions and lemmas here.
The proof uses Galvin's kernel method~\cite{Galvin}, and we reproduce the
relevant definitions and lemmas here.
\begin{definition}
  A \emph{kernel} of a digraph $D$ is an independent set $S$ such that
  for every $v \in D-S$, there is some $w \in S$ with $vw \in E(D)$. A
  digraph is \emph{kernel-perfect} if every induced subgraph has a kernel.
\end{definition}
\begin{definition}
  If $\lz$ is a list assignment on $G$, a \emph{proper $\lz$-coloring}
  of $G$ is a proper coloring $\phi$ of $G$ such that $\psi(v) \in
  \lz(v)$ for all $v \in V(G)$. If $f$ is a function from $V(G)$ into
  the nonnegative integers, we say that $G$ is \emph{$f$-choosable}
  if: for every list assignment $\lz$ with $\sizeof{\lz(v)} \geq f(v)$
  for all $v \in V(G)$, there is a proper $\lz$-coloring of $G$.
\end{definition}
List colorings were first studied by Erd\H os, Rubin, and Taylor~\cite{ERT}
and by Vizing~\cite{Vizing76}.
\begin{lemma}[Bondy--Boppana--Siegel, see \cite{Galvin}]\label{lem:BBS}
  If $D$ is a kernel-perfect orientation of a graph $G$ and $f(v) = 1+d_D^+(v)$
  for all $v \in V(G)$, then $G$ is $f$-choosable.
\end{lemma}
\begin{lemma}[Galvin \cite{Galvin}]\label{lem:galvin}
  Let $H$ be a line graph of an $(X,Y)$-bigraph $U$, and let $\phi$ be an
  integer-valued proper edge coloring of $U$. Define an orientation of $H$
  as follows: for $e_1e_2 \in E(H)$ with $\phi(e_1) < \phi(e_2)$, orient the
  edge $e_1e_2$ from $e_1$ to $e_2$ if their common endpoint lies in $X$,
  and orient the edge from $e_2$ to $e_1$ if their common endpoint lies in $Y$.
  The resulting orientation is kernel-perfect.
\end{lemma}
\begin{proof}[Proof of Lemma~\ref{lem:decomp-sat}]
  Define an auxiliary bipartite graph $U$ as follows. Let $X$ and $Y$
  be disjoint copies of $V(G)$, and for each $v \in V(G)$, let $v_x$
  and $v_y$ denote the copies of $v$ in $X$ and $Y$ respectively.
  Define
  \begin{align*}
    V(U) &= X \cup Y, \\
    E(U) &= \{u_xv_y \st (u,v) \in E(J)\}.
  \end{align*}
  Define an edge-coloring $\phi$ of $U$ as follows: for every arc $(u,v) \in J$,
  let $\phi(u_xv_y)$ be the unique index $i$ such that $(u,v) \in E(H_i)$. Observe
  that $\phi$ is a proper edge-coloring of $U$, since for each $i$, every vertex
  has maximum indegree and maximum outdegree at most $1$ in $H_i$. Furthermore,
  for every $u_x \in X$, the colors on edges incident to $u_x$ are precisely
  the colors $\{1, \ldots, d(u_x)\}$, since if $u_x$ lacks an incident edge
  of color $i$, then $\dout_{H_i}(u) = 0$, which implies $\dout_{H_j}(u) = 0$
  for all $j > i$, so that $u_x$ has no incident edge of color $j$ for any
  $j > i$.

  Let $Z$ be the kernel-perfect orientation of $L(U)$ obtained from
  Lemma~\ref{lem:galvin}. We claim that every edge $u_xv_y \in E(U)$
  has degree at most $d_U(u_x)-1$ in $Z$. To see this, let $e$ be an out-neighbor
  of $u_xv_y$ in $Z$, and let $w$ be the common endpoint of $e$ and $u_xv_y$.
  If $w = u_x$, then $\phi(u_xv_y) < \phi(e) \leq d_U(u_x)$, while if $w = v_y$,
  then $1 \leq \phi(e) < \phi(u_xv_y)$. As $\phi$ is a proper edge-coloring,
  this implies that there are at most $d_U(u_x) - 1$ out-neighbors of $u_xv_y$.
  
  By Lemma~\ref{lem:BBS}, it follows that $L(U)$ is $f$-choosable,
  where $f(u_xv_y) = \dl(u_x)$ for all $u_xv_y \in E(U)$. Now we
  show that $G$ is $\dl$-saturable. Let $\lz$ be any list assignment
  on $G$ with $\sizeof{\lz(v)} \leq \dl(v)$ for all $v$. By adding
  extra colors if necessary, we may assume that
  $\sizeof{\lz(v)} = \dl(v)$ for all $v$.  Define a list assignment
  $\lz'$ on $L(U)$ as follows: for each edge $u_xv_y \in E(U)$, let
  $\lz'(u_xv_y) = \lz(u_x)$. Since $L(U)$ is $f$-choosable, there is a
  proper $\lz'$-coloring $\psi$ of $L(U)$.

  Observe that for each vertex $u \in V(G)$, we have
  $d_Z(u_x) = \dl(u)$, and that all edges incident to $u_x$ in $Z$
  have the list $\lz(u)$ with size $\dl(u)$. Thus, for each color
  $c \in \lz(u)$, the vertex $u_x$ is incident to exactly one edge
  $e$ with $\psi(e) = c$.

  Now for any color $c$ used in $\psi$, let $U_c$ be the subgraph of
  $Z$ consisting of the edges of color $c$, and let $J_c$ be the
  spanning subdigraph of $J$ with arc set
  $\{(u,v) \st u_xv_y \in U_c$. Since $\psi$ is a proper edge-coloring
  of $Z$, we see that $U_c$ is a matching, which implies that $J_c$
  has maximum indegree and maximum outdegree at most $1$. Since $J$
  has no odd cycles, every component of $J_c$ is a directed path
  (possibly a $1$-vertex path) or a directed even cycle. Either way,
  there is a matching $M_c$ in the underlying graph of $J_c$ that
  covers every vertex with positive outdegree in $J_c$.
  
  Now we obtain a partial edge-coloring $\xi$ of $G$ by coloring the
  edges in $M_c$ with color $c$, for each color $c$ used in $\psi$.
  Since each $M_c$ is a matching, the partial edge-coloring is clearly
  proper. Furthermore, if $c \in \lz(u)$ for $u \in V(G)$, then $u_x$
  is incident to some edge $e$ with $\psi(e)=c$, so $u$ has positive
  outdegree in $J_c$, and therefore is covered by $M_c$. It follows
  that $\xi$ is $\lz$-saturating. As $\lz$ was arbitrary, $G$ is
  $\dl$-saturable.
\end{proof}
We prove Theorem~\ref{thm:chordal} via the following lemma. A
\emph{simplicial elimination order} is a vertex ordering $\lob$ such
that when the vertices of $G$ are written $v_1, \ldots, v_n$ in order
according to $\lob$, then for each $i$, the neighborhood of $v_i$ in
the graph $G - \{v_1, \ldots, v_{i-1}\}$ is a clique. A graph is
chordal if and only if it has a simplicial elimination
order~\cite{Chordal}.
\begin{lemma}\label{lem:simp}
  Let $\lob$ be a simplicial elimination ordering on an $n$-vertex graph
  $G$. If $J$ is the orientation of $G$ obtained by orienting
  each edge $uv$ with $u < v$ from $v$ to $u$, then $G$ is
  $\dl$-saturable.
\end{lemma}
\begin{proof}
  We use induction on $n$. When $n=1$, there is nothing to prove, so
  assume that $n > 1$ and the claim holds for smaller $n$. Let $\lob$ be
  a simplicial elimination order on $G$ and let $\lz$ be a list
  assignment with $\sizeof{\lz(v)} = \dl(v)$ for all $v$.

  Let $v$ be the minimum vertex in $\lob$, and let $w_1 < \ldots < w_t$
  be the neighbors of $v$, written in order according to $\lob$. Because
  $N(v)$ is a clique, we have $\dl(w_i) \geq i$ for each $i \in [t]$,
  since $w_i$ has, as smaller neighbors, at least $v$ and $w_1,
  \ldots, w_{i-1}$. Thus, we may choose distinct colors $c_1, \ldots,
  c_t$ such that $c_i \in \lz(w_i)$ for each $i$.

  Let $G' = G-v$ and let $J' = J-v$. Since the restriction of
  $\lob$ to $G'$ is still a simplicial elimination order, and since $J'$
  is obtained from $\lob$ in the prescribed manner, the
  induction hypothesis says that $G'$ is $\dlp$-saturable, where
  $\dlp(w)$ is the downdegree of $w$ in $J'$. Furthermore, for all
  $w \in V(G')$, we have
  \[ \dlp(w) =
  \begin{cases}
    \dl(w) - 1, & \text{if $w \in N(v)$;} \\
    \dl(w), & \text{otherwise.}
  \end{cases} \]
  Let $\lz'$ be the list assignment on $G'$ given by
  \[ \lz'(w) = 
  \begin{cases}
    \lz(w) - c_i, & \text{if $w = w_i$,} \\
    \lz(w), &\text{if $w \notin N(v)$.}
  \end{cases} \]
  Now $\sizeof{\lz'(w)} = \dlp(w)$ for all $w \in V(G)$, so by the
  induction hypothesis, $G'$ has a partial edge coloring $\psi'$ that
  is $\lz'$-down-saturating.

  The partial coloring $\psi'$ is almost $\lz$-saturating in $G$,
  except that each vertex $w_i$ may fail to have $c_i$ on an incident
  edge. (As $v$ has outdegree $0$ in $J$, it makes no demands of the
  coloring.) We extend $\psi'$ to a partial edge coloring of $G$ by
  coloring each edge $vw_i$ with color $c_i$ if $w_i$ does not yet
  have an incident edge of color $c_i$. Since the colors $c_i$ are
  distinct, the resulting partial edge coloring $\psi$ is still
  proper; thus, $\psi$ is $\lz$-saturating. Since $\lz$ was arbitrary,
  it follows that $G$ is $\dl$-saturable.
\end{proof}
Theorem~\ref{thm:chordal} follows immediately from
Lemma~\ref{lem:simp} and Lemma~\ref{lem:satur}: if $D$ is a
$k$-optimal set in a chordal graph $G$, then the induced subgraph
$G[V(G)-D]$ is again chordal, so Lemma~\ref{lem:simp} implies that
$G$ satisfies the hypothesis of Lemma~\ref{lem:satur}.
\bibliographystyle{amsplain} \bibliography{biblio}
\end{document}